\documentclass[reqno, 12pt]{amsart}

\usepackage{amsmath,amssymb,amsthm,amsfonts,mathrsfs}
\usepackage{fullpage}
\usepackage{relsize}
\usepackage{enumerate}
\usepackage{colonequals}
\usepackage[colorlinks=true,
linkcolor=blue,
anchorcolor=blue,
citecolor=red
]{hyperref}
\allowdisplaybreaks 
\newtheorem{theorem}{Theorem}[section]
\newtheorem{lemma}[theorem]{Lemma}

\newtheorem{proposition}[theorem]{Proposition}

\theoremstyle{definition}

\theoremstyle{remark}
\newtheorem{remark}[theorem]{Remark}

\newcommand{\N}{\mathbb{N}}

\newcommand{\R}{\mathbb{R}}
\newcommand{\C}{\mathbb{C}}

\newcommand{\cP}{\mathcal{P}}
\newcommand{\cS}{\mathcal{S}}
\newcommand{\cT}{\mathcal{T}}

\newcommand{\on}{\operatorname}

\renewcommand{\mod}[1]{\,(\on{mod}#1)}
\newcommand{\of}[1]{\left(#1\right)}
\newcommand{\set}[1]{\left\{#1\right\}}
\newcommand{\abs}[1]{\left|#1\right|}

\newcommand{\BEu}[1]{\underset{#1}{\mathlarger{\mathlarger{\mathbb{E}}}^{~}}\,}

\author{Huixi Li}
\address{School of Mathematical Sciences and LPMC, Nankai University, Tianjin 300071, China}
\email{lihuixi@nankai.edu.cn}

\author{Biao Wang}
\address{School of Mathematics and Statistics, Yunnan University, Kunming, Yunnan 650091, China}
\email{bwang@ynu.edu.cn}

\author{Chunlin Wang}
\address{School of Mathematical Sciences, Sichuan Normal University, Chengdu 610064, China}
\email{c-l.wang@outlook.com}

\author{Shaoyun Yi}
\address{School of Mathematical Sciences, Xiamen University, Xiamen, Fujian 361005, China}
\email{yishaoyun926@xmu.edu.cn}

\date{\today}

\makeatletter
\@namedef{subjclassname@2020}{\textup{}2020 Mathematics Subject Classification}
\makeatother

\title[Some ergodic theorems]{Some ergodic theorems over squarefree numbers and squarefull numbers}

\begin{document}
	
\begin{abstract}
In 2022, Bergelson and Richter gave a new dynamical generalization of the prime number theorem by establishing an ergodic theorem along the number of prime factors of integers. They also showed that this generalization holds as well if the integers are restricted to be squarefree. In this paper, we present the concept of invariant averages  under multiplications for arithmetic functions. Utilizing the properties of these invariant averages, we derive several ergodic theorems over squarefree numbers and squarefull numbers. These theorems have significant connections to the Erd\H{o}s-Kac Theorem, the Bergelson-Richter Theorem, and the Loyd Theorem.
\end{abstract}

\maketitle

\section{Introduction and statement of results}

Let $\Omega(n)$ be the number of prime divisors of $n$ counted with multiplicity for any $n\in\N$. Let $\lambda(n)=(-1)^{\Omega(n)}$ be the Liouville function. Then the prime number theorem is equivalent to the assertion that
\begin{equation}\label{eqn_pnt_lambda}
	\lim_{N \to \infty} \frac1N\sum_{n=1}^N \lambda(n)=0,
\end{equation}
see \cite{Landau1953, Mangoldt1897}. Recently, Bergelson and Richter \cite{BergelsonRichter2022} established new dynamical generalizations of the prime number theorem by placing \eqref{eqn_pnt_lambda} in a dynamical framework. Given any uniquely ergodic topological dynamical system $(X, \mu, T)$, they showed that 
\begin{equation}\label{eqn_BR2022thmA}
	\lim_{N\to\infty}\frac1N\sum_{n=1}^N f(T^{\Omega(n)}x)=\int_X f \,d\mu
\end{equation}
holds for any $f\in C(X)$ and $x\in X$. Then the equivalent form \eqref{eqn_pnt_lambda} of the prime number theorem is recovered from  \eqref{eqn_BR2022thmA} by taking $(X,T)$ to be $x\mapsto x +1 \mod 2$ on $X=\{0,1\}$.

Let $\mu(n)$ be the M\"obius function, which is defined to be $1$ if $n = 1$,  $(-1)^k$ if $n$ is the product of $k$ distinct primes, and zero otherwise. Then $\mu^2(n)$ is the indicator function of squarefree numbers, and $\mu(n)=\mu^2(n)\lambda(n)$. It is well known that the natural density of squarefree numbers in $\N$ is equal to $\zeta(2)^{-1}=6/\pi^2$, where $\zeta(s)$ is the Riemann zeta function. If one restricts the numbers $n$ to be squarefree in \eqref{eqn_pnt_lambda} and \eqref{eqn_BR2022thmA}, respectively, then  we obtain another well-known
equivalent form of the prime number theorem
\begin{equation}\label{eqn_pnt_mu}
	\lim_{N \to \infty} \frac1N\sum_{n=1}^N \mu(n)=0,
\end{equation}
and by \cite[Corollary~1.8]{BergelsonRichter2022} we have
\begin{equation}\label{eqn_BR2022cor1.8}
	\lim_{N \to \infty} \frac1N\sum_{\substack{1\le n \le N\\ n\,\text{squarefree}}} f(T^{\Omega(n)}x)=\frac6{\pi^2}\int_X f \,d\mu
\end{equation}
for any $f\in C(X)$ and $x\in X$. Therefore, \eqref{eqn_BR2022cor1.8} is another dynamical generalization of the prime number theorem. 

\subsection{Some ergodic theorems over squarefree numbers}

After Bergelson and Richter's work \cite{BergelsonRichter2022}, there are several generalizations discovered by Loyd \cite{Loyd2022}, Wang \cite{Wang2022}, and Wang et. al. \cite{WWYY2023} on the ergodic theorem \eqref{eqn_BR2022thmA}. A natural question is whether these ergodic theorems hold or not, if the averages are restricted to run over squarefree numbers. In this paper, we will answer this question affirmatively by considering general arithmetic functions of invariant averages. 

Let $a: \N\to\C$ be a bounded arithmetic function. We say that $a(n)$ is of \textit{invariant average under (left) multiplications},  if the average of $a(n)$
\begin{equation}\label{eqn_inv_ave1}
    \lim_{N\to\infty}\frac1N\sum_{n=1}^N a(n)
\end{equation}
exists and satisfies 
\begin{equation}\label{eqn_inv_ave2}
    \lim_{N\to\infty}\frac1N\sum_{n=1}^N a(mn)= \lim_{N\to\infty}\frac1N\sum_{n=1}^N a(n)
\end{equation}
for all $m\in \N$. For example, the orbits $a(n)=f(T^{\Omega(n)}x)$ in \eqref{eqn_BR2022thmA} are of invariant average under multiplications. The following theorem is our main result on such invariant averages.

\begin{theorem}
\label{mainthm_sqfree}
Let $a: \N\to\C$ be a bounded arithmetic function of average $A$. Let $\cS$ be a finite set of primes.  If the average of $a(n)$ is invariant under multiplications, then
\[
	\lim_{N\to\infty}\frac1N\sum_{\substack{1\leq n \leq N\\ n\,\text{squarefree}\\ p\,\nmid\,n, \forall p \in \cS}} a(n)= \frac{\alpha(\cS)}{\zeta(2)} A,
\]where $\alpha(\cS)=\prod_{p\in \cS}\frac{p}{p+1}$. In particular, if $\cS=\emptyset$, then we have
\[
	\lim_{N\to\infty}\frac1N\sum_{\substack{1\leq n \leq N\\ n\,\text{squarefree}}} a(n)= \frac{6}{\pi^2}A.
\]
\end{theorem}

Clearly, \eqref{eqn_BR2022cor1.8} follows by Theorem~\ref{mainthm_sqfree} if we take $a(n)=f(T^{\Omega(n)}x)$ as in \eqref{eqn_BR2022thmA}. Moreover, as applications of Theorem~\ref{mainthm_sqfree}, the following analogies of \eqref{eqn_BR2022cor1.8} hold as well.

\begin{theorem}\label{mainthm_sqfree_applications}
Let $\cS$ be a finite set of primes, and let $\alpha(\cS)=\prod_{p\in \cS}\frac{p}{p+1}$.

\begin{enumerate}
	\item Let $(X, \mu, T)$ be a uniquely ergodic topological dynamical system. For $k\ge0$, let $\varphi_k(n)$ be the $k$-fold iterate of Euler's totient function $\varphi$, i.e., $\varphi_0(n)=n$ and $\varphi_k(n)=\varphi_{k-1}(\varphi(n))$ for $k\ge1$. Let 
 \begin{equation}\label{akbk}
     a_k=\frac{1}{(k+1)!} \quad \text{and} \quad b_k=\frac{k!}{\sqrt{2k+1}}.
     \end{equation}
 Let $\cT$ be a set of primes of natural density $\delta(\cT)$. Denote by $p_{\max}(n)$ the largest prime factor of $n$ with $p_{\max}(1)=1$. Let $k\ge0$ and $F \in C_c(\R)$, then  we have
\begin{equation}\label{eqn_maincor_bkw}
    \begin{split}
      &\lim_{N\to\infty}\frac1N\sum_{\substack{1\leq n \leq N\\ n\,\text{squarefree}\\ p\,\nmid\,n, \forall p \in \cS\\p_{\max}(n)\in \cT}} F\Big(\frac{\Omega(\varphi_k(n)) - a_k(\log \log N)^{k+1} }{b_k(\log \log N)^{k+1/2}} \Big)  f(T^{\Omega(n) }x)\\
	&\qquad\qquad\qquad\qquad = \frac{\alpha(\cS)\delta(\cT)}{\zeta(2)}\Big(\frac{1}{\sqrt{2\pi}} \int_{-\infty}^{\infty} F(t) e^{-t^2/2} \, dt\Big)\Big(\int_X f \, d\mu \Big)   
    \end{split}
\end{equation}
for any $f\in C(X)$ and $x\in X$.	
	
	\item Let $(Y,\nu, S)$ be a multiplicative, finitely generated and strongly uniquely ergodic topological dynamical system,  then we have
\begin{equation}\label{eqn_maincor_BRthmB}
	\lim_{N\to\infty}\frac1N\sum_{\substack{1\leq n \leq N\\ n\,\text{squarefree}\\ p\,\nmid\,n, \forall p \in \cS}} g(S_ny)	= \frac{\alpha(\cS)}{\zeta(2)}  \int_Y g\,d\nu
\end{equation}
for any $g\in C(Y)$ and $y\in Y$. 

In particular, in a uniquely ergodic system $(X, \mu, T)$, for any $m\in \N, 0\leq r\leq m-1$, $f\in C(X)$ and $x\in X$, we have
\begin{equation}\label{eqn_maincor_BRcor1.16}
	\lim_{N \to \infty} \frac1N\sum_{\substack{1\le n \le N\\ n\,\text{squarefree}}} f(T^{\Omega(mn+r)}x)=\frac6{\pi^2}\int_X f \,d\mu.
\end{equation}

\end{enumerate}	
\end{theorem}

\begin{remark}
One may consider $r$-free numbers and the case $\cS$ is infinite for $r\geq2$ in Theorem~\ref{mainthm_sqfree} and Theorem~\ref{mainthm_sqfree_applications}; see \cite{EvelynLinfoot1931, CellarosiVinogradov2013, Brown2021, Brown2023}. And one may also consider the analogues of these results over finite fields.  
\end{remark}

Let $\cP$ denote the set of all primes. As demonstrated in Table \ref{table_results_over_squarefree_numbers}, by applying Theorem \ref{mainthm_sqfree_applications} with specific choices of $k$, $\cS$, $\cT$, $F$, and $f$, we can derive several known ergodic theorems concerning squarefree numbers. For instance, as indicated in row 2, setting $k=0$, $\cS=\emptyset$, and $\cT=\cP$ in equation \eqref{eqn_maincor_bkw} yields Loyd's result \cite[Theorem 1.3]{Loyd2022} for squarefree numbers.

\begin{table}[h]
\begin{equation*}
\begin{array}{c|c|c|c|c|c|l}
\text{Equation}&  k  & \cS & \cT & F & f & \text{Theorems}\\\hline
\eqref{eqn_maincor_bkw} & 0  & \emptyset &  \cP & & 1 & \text{Erd\H{o}s-Kac Theorem \cite{ErdosKac1940}} \\\hline
\eqref{eqn_maincor_bkw} & 0  & \emptyset &  \cP & & & \text{Loyd \cite[Theorem~1.3]{Loyd2022}} \\\hline
\eqref{eqn_maincor_bkw} & \text{\footnotesize NA} & \emptyset & & 1 & & \text{Wang \cite[Eq.~(7)]{Wang2022}} \\\hline
\eqref{eqn_maincor_bkw} & 1 & \emptyset & & & & \text{Wang, Wei, Yan, and Yi \cite[Theorem~1.4]{WWYY2023}}  \\\hline\hline
\eqref{eqn_maincor_BRthmB} & \text{\footnotesize NA} & \emptyset & \text{\footnotesize NA} & \text{\footnotesize NA} & & \text{Bergelson and Richter \cite[Theorem~B]{BergelsonRichter2022}}\\\hline
\eqref{eqn_maincor_BRcor1.16} & \text{\footnotesize NA} & \emptyset & \text{\footnotesize NA} & \text{\footnotesize NA} & & \text{Bergelson and Richter \cite[Corollary~1.16]{BergelsonRichter2022}}\\\hline
\end{array}    
\end{equation*}    
\caption{Associated results \textit{over squarefree numbers}}\label{table_results_over_squarefree_numbers}
\end{table}

\subsection{Some ergodic theorems over squarefull numbers}

Next, we consider the restriction of Bergelson-Richter's theorem \eqref{eqn_BR2022thmA} over squarefull numbers. In general, we study the $k$-full numbers for any integer $k\ge2$. A natural number $n$ is said to be \textit{$k$-full} if $p^k$ is a divisor of $n$ for every prime factor $p$ of $n$. When $k=2$, $2$-full numbers are called  \textit{squarefull numbers} or \textit{powerful numbers}, see \cite{Golomb1970}. The distribution and properties of squarefull numbers and $k$-full numbers have been intensely studied in the literature, see \cite{ErdosSzekeres1934, BatemanGrosswald1958, Ivic1978, Trifonov2002, Blomer2005, Chan2015, Chan2023, BBC2024} and so on. Similar to the case on squarefree numbers, we study the summation of bounded arithmetic functions over $k$-full numbers first. For a bounded function $a: \N\to\C$, we say $a(n)$ is of \textit{$k$-th power invariant average under multiplications}, if the following mean value
\begin{equation}\label{eqn_k_invariant_cond_1}
	\lim_{N\to\infty}\frac1N\sum_{n=1}^N a(n^k)
\end{equation}
along $k$-th powers exists and 
\begin{equation}\label{eqn_k_invariant_cond_2}
    \lim_{N\to\infty}\frac1N\sum_{n=1}^N a(n^km)= \lim_{N\to\infty}\frac1N\sum_{n=1}^N a(n^k)
\end{equation}
holds for all $m\in \N$. For example, let $S$ be a set of primes of natural density $\delta(S)$ within all primes, then by \cite[Lemma 4.5]{Wang2022} and \cite[Theorem 3.1]{KuralMcDonaldSah2020} the indicator function $1_{p_{\max}(n)\in S}$ is of $k$-th power invariant average $\delta(S)$ under multiplications for all $k\ge2$. If we take $k=1$ in \eqref{eqn_k_invariant_cond_1} and \eqref{eqn_k_invariant_cond_2}, then it is as the same definition of invariant average under multiplications as in \eqref{eqn_inv_ave1} and \eqref{eqn_inv_ave2}.  Let $(X, \mu, T)$ be a uniquely ergodic topological dynamical system. If $(X,\mu, T^n)$ is also uniquely ergodic for every $n\in \N$, then we say $(X, \mu, T)$ is a \textit{totally uniquely ergodic system}.  We will show in the proof of Theorem~\ref{mainthm_kfull_applications} that if $(X, \mu, T)$ is a totally uniquely ergodic system, then the sequence $\set{f(T^{\Omega(n)}x): n\in\N}$ in \eqref{eqn_BR2022thmA} is of $k$-th power invariant average under multiplications for any $k\ge2$.  

If $B$ is a finite nonempty set, we define $\BEu{x\in B}f(x):=\frac1{|B|}\sum_{x\in B}f(x)$ for any function $f:B\to\C$ on $B$. For a real number $s\geq 1$,  let $[s]\colonequals\N\cap[1,s]$ be the set of natural numbers between $1$ and $s$. Then $[N]=\set{1,\dots,N}$ for any $N\in \N$.  Using this notation and the fact that density of squarefree numbers in $\N$ is $6/\pi^2$, \eqref{eqn_BR2022cor1.8} can be rewritten as
\[
	\lim_{N\to\infty}\BEu{\substack{n \in [N]\\ n \, \text{squarefree}}} f(T^{\Omega(n)}x)=\int_X f \,d\mu.
\]
The following theorem is our main result on  $k$-th power invariant averages.

\begin{theorem}
\label{mainthm_kfull}
Let $k\geq2$. If $a: \N\to\C$ is a bounded arithmetic function of $k$-th power invariant average under multiplications, then we have
\[
\lim_{N\to\infty} \BEu{\substack{ n \in [N]\\ n \,\text{is} \, k\text{-full}}} a(n)= \lim_{N\to\infty}\BEu{n\in [N] } a(n^k).
\]
\end{theorem}

As an application of Theorem~\ref{mainthm_kfull}, we can establish analogues of Bergelson-Richter's theorem, Erd\H{o}s-Kac theorem  and Loyd's theorem for  $k$-full numbers. 

\begin{theorem}\label{mainthm_kfull_applications}

Let $(X, \mu, T)$ be a totally uniquely ergodic  system, then the following statements hold:

\begin{enumerate}
	\item For any $f\in C(X)$ and $x\in X$, we have
 \begin{equation}\label{eqn_mainthm_ergogic}
	 	\lim_{N\to\infty} \BEu{\substack{ n \in [N]\\ n \,\text{is} \, k\text{-full}}} f(T^{\Omega(n) }x) = \int_X f \, d\mu.
	 \end{equation}
	 
	 \item For any $F\in C_c(\R)$, we have 
	\begin{equation}\label{eqn_mainthm_EK}
	\lim_{N \to \infty}  \BEu{\substack{ n \in [N]\\ n \,\text{is} \, k\text{-full}}} F\Big( \frac{\Omega(n) - k\log \log N }{k\sqrt{\log \log N}} \Big)=\frac{1}{\sqrt{2\pi}} \int_{-\infty}^{\infty} F(t) e^{-t^2/2} \, dt.
\end{equation}

\item For any $f\in C(X)$, $x\in X$ and $F\in C_c(\R)$, we have
	\begin{equation}\label{eqn_mainthm_Loyd}
\lim_{N \to \infty}  \BEu{\substack{ n \in [N]\\ n \,\text{is} \, k\text{-full}}} F \Big( \frac{\Omega(n) - k\log \log N }{k\sqrt{\log \log N}} \Big) f(T^{\Omega(n) }x)
	= \Big(\frac{1}{\sqrt{2\pi}} \int_{-\infty}^{\infty} F(t) e^{-t^2/2} \, dt\Big)\Big( \int_X f \, d\mu \Big).
\end{equation}
	 \end{enumerate}
\end{theorem}

\begin{remark}
Recently, Donoso, Le, Moreira and  Sun \cite[Theorem~D]{DLMS2024} showed an analogue of Bergelson and Richter's theorem \eqref{eqn_BR2022thmA} along $\Omega(m^2+n^2)$: let $(X,T, \mu)$ be a uniquely ergodic topological dynamical system, then for any $f\in C(X)$ and $x\in X$, we have
	\begin{equation}\label{eqn_DLMS2024thmD}
		\lim_{N\to\infty} \BEu{ m,n\in [N]} f(T^{\Omega(m^2+n^2)}x)=\int_Xfd\mu.
	\end{equation}
In \cite{Wang2024}, the second author established a variant of \eqref{eqn_DLMS2024thmD} over co-prime integer pairs: for any $x\in X$ and any $f\in C(X)$, we have
\[
		\lim_{N\to\infty} \BEu{\substack{ m,n\in [N]\\ \gcd(m,n)=1}} f(T^{\Omega(m^2+n^2)}x)=\int_Xfd\mu.
\]
\end{remark}

\subsection{An application to Richter's generalization of the prime number theorem}

In 2021, Richter \cite{Richter2021} gave a new elementary proof of the prime number theorem and showed that for any bounded $a:\N\to\C$ one has
\begin{equation}\label{eqn_Richter}
	\frac1N\sum_{n\leq N}a(\Omega(n))= \frac1N\sum_{n\leq N}a(\Omega(n)+1)+o_{N\to\infty}(1).
\end{equation}
Similar to Theorems~\ref{mainthm_sqfree} and ~\ref{mainthm_kfull}, by the arguments in their proofs, we obtain the following analogues of  \eqref{eqn_Richter} with respect to squarefree numbers and $k$-full numbers.

\begin{theorem}\label{thm_Richter_sqfree}
	For any bounded function $a:\N\to\C$, we have
\begin{equation}\label{eqn_Richter_sqfree}
	\frac1N\sum_{\substack{n\leq N\\ n\,\text{squarefree}}}a(\Omega(n))= \frac1N\sum_{\substack{n\leq N\\ n\,\text{squarefree}}}a(\Omega(n)+1)+o_{N\to\infty}(1).
\end{equation}
Moreover, for any integer $k\ge2$, we have
\begin{equation}\label{eqn_Richter_kfull}
	\frac1{N^{1/k}}\sum_{\substack{n\leq N\\ n \,is\, k\text{-full}}}a(\Omega(n))= \frac1{N^{1/k}}\sum_{\substack{n\leq N\\ n \,is\, k\text{-full}}}a(\Omega(n)+k)+o_{N\to\infty}(1).
\end{equation}
\end{theorem}

Taking $a(n)=(-1)^n$ in \eqref{eqn_Richter} and \eqref{eqn_Richter_sqfree} gives \eqref{eqn_pnt_lambda} and \eqref{eqn_pnt_mu}, respectively. Hence Theorem~\ref{thm_Richter_sqfree} is a generalization of the prime number theorem as well. 

In Section~\ref{sec_pre}, we will introduce  the statements of the theorems listed in Table~\ref{table_results_over_squarefree_numbers}, including some background on uniquely ergodic topological dynamical systems. Some basic facts on $k$-full numbers are also introduced. Then in Section~\ref{sec_pf_mainthm}, we will apply the technique in \cite[Theorem~6]{JiangLiu2024} to prove Theorem~\ref{mainthm_sqfree}. In Section~\ref{sec_pf_maincor}, we will prove Theorem~\ref{mainthm_sqfree_applications} by showing that the averages in \eqref{eqn_maincor_bkw} and \eqref{eqn_maincor_BRthmB} are invariant under multiplications and applying Theorem~\ref{mainthm_sqfree}. In Sections~\ref{sec_mainthm_proof} and \ref{sec_proofs}, we will prove the theorems on $k$-full numbers using the arguments similar to the proofs of Theorems~\ref{mainthm_sqfree} and \ref{mainthm_sqfree_applications}. Finally, we will apply the methods in the proof of Theorems~\ref{mainthm_sqfree} and ~\ref{mainthm_kfull} to prove Theorem~\ref{thm_Richter_sqfree} in the last section.

\section{Preliminaries}
\label{sec_pre}

\subsection{Uniquely ergodic topological dynamical system}

Let $X$ be a compact metric space, let $T\colon X\to X$ be a continuous map, and let $C(X)$ be the set of continuous functions defined on $X$. Since $T^m\circ T^n=T^{m+n}$ for any $m, n\in \N$, the transformation $T$ naturally induces an action of $(\N, +)$ on $X$. We call the pair $(X, T)$ an \textit{additive topological dynamical system}. A Borel probability measure $\mu$ on $X$ is said to be \textit{$T$-invariant} if $\mu(T^{-1}A)=\mu(A)$ for all Borel measurable subsets $A\subset X$. By the Bogolyubov-Krylov Theorem (see for example \cite[Corollary~6.9.1]{Walters1982}), every additive
topological dynamical system $(X, T)$ possesses at least one $T$-invariant Borel probability
measure. If $(X, T)$ admits only one such measure, then it is called \textit{uniquely ergodic}. It is well-known  that $(X, \mu, T)$ is uniquely ergodic if and only if 
\[
	\lim_{N \to \infty} \frac1N\sum_{n=1}^N  f(T^n x) = \int_X f \, d\mu
\]
holds for  all $x \in X$ and $f\in C(X)$, see for example \cite[Theorem~6.19]{Walters1982}. 

A uniquely ergodic system $(X, \mu, T)$ is called a \textit{totally uniquely ergodic system}, if $(X,\mu, T^n)$ is uniquely ergodic for every $n\in \N$. Given a real number $\alpha\in \R$, let $X = [0,1]$, endowed with the Lebesgue measure $\mu$, and let $T_\alpha : X\to X$ be the map defined by $T_\alpha x=x+\alpha \mod{1}$. This system is called a \textit{circle rotation}. Then the system $(X, \mu, T_\alpha)$ is totally uniquely ergodic if and only if $\alpha$ is irrational. Let $X=\{0,1\}$ be a two-point system with $\mu(\set{0})=\mu(\set{1})=1/2$ and $T: x\mapsto x+1 \mod{2}$. Then this system is uniquely ergodic but not totally uniquely ergodic. 

A \textit{multiplicative topological dynamical system} is a pair $(Y, S)$, where $Y$ is a compact metric space and $S=(S_n)_{n\in N}$ is a sequence of continuous maps on $Y$ satisfying $S_{nm}=S_n\circ S_m$ for all $m, n\in \N$. That is, $S$ is an action of $(\N,\cdot)$ on $Y$ by continuous maps.  We say that $S$ is \textit{finitely generated} if the set $\set{S_p: p\in \cP}$ of generators of $S$ is finite. And $(Y, S)$ is called \textit{strongly uniquely ergodic} if there is only one Borel probability measure on $Y$, say  $\nu$, satisfying that $\nu$ is invariant under $S_p, \forall p\in P$ for  some set of primes $P\subset \cP$ with $\sum_{p\notin P}1/p<\infty$. Using combinatorial techniques, Bergelson and Richter established the following ergodic theorem on finitely generated and strongly uniquely ergodic multiplicative topological dynamical systems. 

\begin{theorem}[{\cite[Theorem~B]{BergelsonRichter2022}}]
\label{thm_BRThmB}
Let $(Y,\nu, S)$ be a multiplicative, finitely generated and strongly uniquely ergodic topological dynamical system. Then 
\[
	\lim_{N\to\infty}\frac1N\sum_{n=1}^N g(S_ny)	= \int_Y g\,d\nu
\]
for every $y\in Y$ and $g\in C(Y)$.
\end{theorem}

In particular, by Theorem~\ref{thm_BRThmB} they proved the following generalization of \eqref{eqn_BR2022thmA} along arithmetic progressions. 

\begin{theorem}[{\cite[Corollary~1.16]{BergelsonRichter2022}}]
    Let $(X, \mu, T)$ a uniquely ergodic system, then we have
\[
	\lim_{N\to\infty}\frac1N\sum_{n=1}^N f(T^{\Omega(mn+r)}x)=\int_X f \,d\mu
\]
for any $m\in \N, 0\leq r\leq m-1$, $f\in C(X)$ and $x\in X$.
\end{theorem}

\subsection{Erd\H{o}s-Kac Theorem}

The well-known Erd\H{o}s-Kac Theorem \cite{ErdosKac1940} says that the number of prime divisors $\Omega(n)$ of $n$ satisfies the following Gaussian distribution:

\begin{theorem}[Erd\H{o}s-Kac Theorem]
\label{thm_EK}
    Let $C_c(\R)$ denote the set of compactly supported continuous functions on $\R$, then we have 
	\begin{equation}\label{eqn_EK_origin}
	\lim_{N \to \infty} \frac1N\sum_{n=1}^N F \Big( \frac{\Omega(n) - \log \log N }{\sqrt{\log \log N}} \Big)=\frac{1}{\sqrt{2\pi}} \int_{-\infty}^{\infty} F(t) e^{-t^2/2} \, dt
\end{equation}
for any $F\in C_c(\R)$.
\end{theorem}

This result is the start of the probabilistic number theory. Moreover, $\Omega(n)$ is well-distributed in residue classes. By \eqref{eqn_pnt_lambda},  the prime number theorem is equivalent to the assertion that $\Omega(n)$ is evenly distributed modulo 2.  In general, by the Pillai-Selberg Theorem \cite{Pillai1940, Selberg1939} the set $\set{n\in \N: \Omega(n)\equiv r \mod m}$ has natural density equal to $1/m$ for all $m\in\N$ and all $r\in \set{0,1,\dots, m-1}$. In 1985, Erd\H{o}s and Pomerance \cite{ErdosPomerance1985} proved the following Erd\H{o}s-Kac type theorem for $\Omega(\varphi(n))$:

\begin{theorem}[Erd\H{o}s-Pomerance Theorem]
\label{thm_EP}
We have
\[
	\lim_{N \to \infty} \frac1N\sum_{n=1}^N F\Big( \frac{\Omega(\varphi(n)) - \frac12(\log \log N)^2 }{\frac1{\sqrt3}(\log \log N)^{3/2}} \Big)
	=
	\frac{1}{\sqrt{2\pi}} \int_{-\infty}^{\infty} F(t) e^{-t^2/2} \, dt
\]
for any $F\in C_c(\R)$. 
\end{theorem}

Let $\varphi_k(n)$ be the $k$-fold iterate of Euler's totient function $\varphi$, $k\ge1$. In 1991, K\'atai \cite{Katai1991} proved an Erd\H{o}s-Kac type theorem for $\Omega(\varphi_2(n))$. Furthermore, in 1997, Bassily, K\'atai, and Wijsmuller \cite{BKW1997} proved the following Erd\H{o}s-Kac type theorem for $\Omega(\varphi_k(n))$:

\begin{theorem}[{\cite[Theorem~1]{BKW1997}}] 
\label{thm_BKW}
	Let $k\ge1$. For any $F\in C_c(\R)$, we have 
\[
	\lim_{N \to \infty} \frac1N\sum_{n=1}^N F\Big( \frac{\Omega(\varphi_k(n)) - a_k(\log \log N)^{k+1} }{b_k(\log \log N)^{k+1/2}} \Big)
	=
	\frac{1}{\sqrt{2\pi}} \int_{-\infty}^{\infty} F(t) e^{-t^2/2} \, dt,
\]
where $a_k$ and $b_k$ are defined as in \eqref{akbk}.
\end{theorem}

We refer readers to \cite{MMP2023} for more results on the Erd\H{o}s-Kac Theorem.

\subsection{Loyd Theorem}

Two sequences $a, b: \N \to \C$ are called \textit{asymptotically uncorrelated} if 
\begin{equation}
    \lim_{N \to \infty} \frac1N\sum_{n=1}^N a(n) \overline{b(n)} = \bigg( \lim_{N \to \infty} \frac1N\sum_{n=1}^N a(n) \bigg)\bigg( \lim_{N \to \infty} \frac1N\sum_{n=1}^N \overline{b(n)} \bigg).
\end{equation}

Recently, Loyd \cite{Loyd2022} shows that the sequences appearing in the Erd\H{o}s-Kac Theorem \eqref{eqn_EK_origin} and the Bergelson-Richter Theorem \eqref{eqn_BR2022thmA} are asymptotically uncorrelated:

\begin{theorem}[{\cite[Theorem~1.3]{Loyd2022}}]
Let $(X, \mu, T)$ be a uniquely ergodic topological dynamical system. Then we have
	\begin{equation}\label{eqn_Loyd2022_origin}
\lim_{N \to \infty} \frac1N\sum_{n=1}^N F \Big( \frac{\Omega(n) - \log \log N }{\sqrt{\log \log N}} \Big) f(T^{\Omega(n) }x)
	=\Big(\frac{1}{\sqrt{2\pi}} \int_{-\infty}^{\infty} F(t) e^{-t^2/2} \, dt\Big)\Big( \int_X f \, d\mu \Big)
\end{equation}
for any $F\in C_c(\R), f\in C(X)$ and $x\in X$.
\end{theorem}

Inspired by Loyd's work in \cite{Loyd2022} and the second author's work \cite{Wang2022}, Wang et. al. \cite{WWYY2023} proved a similar result for the Erd\H{o}s-Pomerance Theorem. 
\begin{theorem}[{\cite[Theorem~1.4]{WWYY2023}}]
	\label{thm_WWYY2023_origin}
	Let $(X, \mu, T)$ be a uniquely ergodic topological dynamical system, let $\cT$ be a set of primes of natural density $\delta(\cT)$, and let $F \in C_c(\R)$. Then we have
	\begin{equation}\label{eqn_WWYY2023_origin}
			\begin{aligned}
		&\lim_{N \to \infty}\frac1N\sum_{\substack{1\le n\le N\\p_{\max}(n)\in \cT}} F\Big( \frac{\Omega(\varphi(n)) - \frac12(\log \log N)^2 }{\frac1{\sqrt3}(\log \log N)^{3/2}} \Big) f(T^{\Omega(n) }x) \\
		&\qquad =\delta(\cT)\Big(\frac{1}{\sqrt{2\pi}} \int_{-\infty}^{\infty} F(t) e^{-t^2/2} \, dt\Big)\Big( \int_X f \, d\mu \Big)
		\end{aligned}
	\end{equation}
	for all $f \in C(X)$ and $x \in X$. 
\end{theorem}
In particular, taking $F=1$ in \eqref{eqn_WWYY2023_origin} gives the second author's result in \cite[Eq.~(7)]{Wang2022}.

\subsection{\texorpdfstring{$k$}{k}-full numbers} 
Every $k$-full number $n$ can be written uniquely as
$$n=m^k n_1^{k+1}\cdots n_{k-1}^{2k-1},$$
where $n_1,\dots, n_{k-1}$ are squarefree and pairwise coprime, see for example \cite{Ivic1978}. In particular, every squarefull number $n$  can be written uniquely as $n=m^2l^3$, where $l$ is squarefree. Let $Q_k(N)$ be the number of $k$-full positive integers not exceeding $N$.  In 1935, Erd\H{o}s and Szekeres \cite{ErdosSzekeres1934} first proved an asymptotic formula of $Q_k(N)$ as follows
\begin{equation}\label{eqn_ErdosSzekeres1934}
	Q_k(N)=c_kN^{\frac1k} +O(N^{\frac1{k+1}}),
\end{equation}
where $c_k=\prod_{p}(1+\sum_{m=k+1}^{2k-1}p^{-m/k})$. We will use \eqref{eqn_ErdosSzekeres1934} in the proof of Theorem~\ref{mainthm_kfull}. 

For the squarefull numbers, in 1958 Bateman and Grosswald \cite{BatemanGrosswald1958} showed the hitherto best unconditional result
\begin{equation}\label{eqn_BatemanGrosswald1958}
	Q_2(x)=Ax^{\frac12}+Bx^{\frac13}+O(x^{\frac16}),
\end{equation}
where $A=\zeta(3/2)/\zeta(3), B=\zeta(2/3)/\zeta(2)$. Under the Riemann Hypothesis, a number of improvements have been made on the error term in \eqref{eqn_BatemanGrosswald1958}, see \cite{Liu2016} for the records.  On other aspects of the distribution of the squarefull numbers, readers may refer to a recent work \cite{BBC2024} for arithmetic progressions in squarefull numbers and \cite{Chan2023} for the short intervals results.

\section{Proof of Theorem~\ref{mainthm_sqfree}}
\label{sec_pf_mainthm}

In this section, we prove Theorem~\ref{mainthm_sqfree}. Let $\cS\subseteq\cP$ be a finite set of primes. We denote by $w_\cS$ the indicator of the integers not divisible by any primes in $\cS$, i.e., 
\[     w_\cS(n)=\begin{cases}
 	0& \text{if } p\mid n \text{ for some } p\in \cS,\\
 	1& \text{otherwise}. 
 \end{cases}
\]
  If $\cS=\emptyset$ is an empty set, then we set $w_\cS(n)=1$ for all $n\in \N$. For $w_\cS(n)$, we have the following properties, see \cite[Lemmas~A.5-7, Corollary~A.8]{CellarosiSinai2013}.
 
 \begin{lemma}
 	We have
 	\begin{align}
 		\mu^2(n)w_\cS(n)&=\sum_{d^2|n} \mu(d)w_\cS(d)w_\cS(n/d),\label{CS13A5}\\
 		\sum_{n=1}^\infty \frac{w_\cS(n)}{n^2}&=\beta(\cS)\zeta(2),\nonumber\\
 		(\mu w_\cS)\ast w_\cS&=1_{n=1},\nonumber\\
	\sum_{n=1}^\infty \frac{\mu(n)w_\cS(n)}{n^2}&=\frac1{\beta(\cS)\zeta(2)}, \label{eqn_ws_zeta2}
 	\end{align}
 where $\beta(\cS)=\prod_{p\in \cS}\frac{p^2-1}{p^2}$. 	Moreover, we have that
\[
 		\frac1N\sum_{n=1}^N \mu^2(n)w_\cS(n) =\frac{\alpha(\cS)}{\zeta(2)}+O_\cS(N^{-\frac12}),
\]
where $\alpha(\cS)=\prod_{p\in \cS}\frac{p}{p+1}$.
\end{lemma}

Now, we apply the above properties to estimate summations over squarefree numbers. 
\begin{proposition}\label{prop_sqfree}
	Let $a: \N\to\C$ be  bounded. Let $N\ge1$. Then for any $1\leq D \leq \sqrt{N}$, we have
	\begin{equation}\label{eqn_keyprop}
	\frac{1}{N}\sum_{n=1}^N \mu^2(n)w_\cS(n) a(n) = \sum_{d=1}^D \frac{\mu(d)w_\cS(d)}{d^2}\BEu{n\in [\frac{N}{d^2}]}w_\cS(n)a(d^2n)+O\Big(\frac{1}{D}\Big)+O\Big(\frac{1}{\sqrt{N}}\Big).
	\end{equation}
\end{proposition}

\begin{proof} Without loss of generality, we may assume that $|a(n)|\leq1, \forall n\in \N$. Then by \eqref{CS13A5}
\begin{align}
	&\quad\frac{1}{N}\sum_{n=1}^N \mu^2(n)w_\cS(n) a(n) \nonumber\\
 &= \frac{1}{N}\sum_{n=1}^N \sum_{d^2|n}\mu(d)w_\cS(d)w_\cS(n/d) a(n) \nonumber\\
	&=\frac{1}{N}\sum_{1\leq d^2m \leq N}  \mu(d)w_\cS(d)w_\cS(dm) a(d^2m) \nonumber\\
	&=\frac{1}{N} \sum_{d=1}^{\sqrt{N}} \sum_{m=1}^{\frac{N}{d^2}}  \mu(d)w_\cS(d)w_\cS(m) a(d^2m) \nonumber\\
	&=\sum_{d=1}^{\sqrt{N}}  \mu(d)w_\cS(d)  \left(\frac{1}{N} \sum_{m=1}^{\frac{N}{d^2}}  w_\cS(m) a(d^2m)\right) \nonumber\\
	&=\sum_{d=1}^{\sqrt{N}} \mu(d) w_\cS(d) \left(\frac1{d^2}\BEu{n\in [\frac{N}{d^2}]}w_\cS(n) a(d^2n)+O\Big(\frac{1}{N}\Big)\right)\nonumber\\
	&=\sum_{d=1}^{\sqrt{N}} \frac{\mu(d)w_\cS(d) }{d^2}\BEu{n\in [\frac{N}{d^2}]}w_\cS(n)a(d^2n)+O\Big(\frac{1}{\sqrt{N}}\Big)\nonumber\\
	&=\Big(\sum_{d=1}^{D} + \sum_{d=D+1}^{\sqrt{N}}\Big)\frac{\mu(d)w_\cS(d) }{d^2}\BEu{n\in [\frac{N}{d^2}]}w_\cS(n)a(d^2n) +O\Big(\frac{1}{\sqrt{N}}\Big).\nonumber
\end{align}
For any $1\leq D \leq \sqrt{N}$, by $\sum_{d=D+1}^\infty \frac1{d^2}\leq \frac1D$ and $\abs{\BEu{n\in[\frac{N}{d^2}]} w_\cS(n)a(d^2n)}\leq1$, we get that 
\begin{align}
  \abs{\sum_{d=D+1}^{\sqrt{N}} \frac{\mu(d)w_\cS(d) }{d^2}\BEu{n\in [\frac{N}{d^2}]}w_\cS(n)a(d^2n)}&\leq  \sum_{d=D+1}^\infty \frac{1}{d^2}\abs{\BEu{n\in [\frac{N}{d^2}]}w_\cS(n)a(d^2n)}\nonumber\\
  &\leq \sum_{d=D+1}^\infty \frac1{d^2}\leq \frac1D.  \nonumber
\end{align}
Thus, \eqref{eqn_keyprop} follows.
\end{proof}

\noindent\textit{Proof of Theorem~\ref{mainthm_sqfree}}. Suppose $a: \N\to\C$ is a bounded arithmetic function of invariant average $A$ under multiplications, we first show that 
\begin{equation}\label{eqn_keythm1}
	\quad\lim_{N\to\infty}\frac1N\sum_{n=1}^N w_\cS(n)a(n)= \alpha_0(\cS)  A,
\end{equation}
where $\alpha_0(\cS)=\prod_{p\in\cS}\frac{p-1}{p}$.

Put $P(\cS)= \prod_{p\in \cS}p$. Then $w_\cS(n)=1_{(n,P(\cS))=1}=\sum_{d|(n,P(\cS))}\mu(d)$. It follows that
\begin{align}
	\frac1N\sum_{n=1}^N w_\cS(n)a(n)&= \frac1N\sum_{n=1}^N \sum_{d|(n,P(\cS))} \mu(d) a(n) \nonumber\\
	&=\frac1N \sum_{d|P(\cS)} \mu(d) \sum_{n=1}^{\frac{N}{d}}a(dn) \nonumber\\
	&=\sum_{d|P(\cS)} \frac{\mu(d)}{d}\left(\frac{1}{\frac{N}d}\sum_{n=1}^{\frac{N}{d}}a(dn)\right).\nonumber
\end{align}
Since the average of $a(n)$ is invariant under left multiplications, we get that
\begin{equation*}
    \lim_{N\to\infty}\frac1N\sum_{n=1}^N w_\cS(n)a(n)=\sum_{d|P(\cS)} \frac{\mu(d)}{d}\cdot \lim_{N\to\infty} \frac{1}{\frac{N}d} \sum_{n=1}^{\frac{N}{d}}  a(dn) =\sum_{d|P(\cS)} \frac{\mu(d)}{d}\cdot A= \alpha_0(\cS)A,
\end{equation*}
which gives \eqref{eqn_keythm1} as claimed. 

Observe that, for any $m\in \N$, the average of $a(mn)$ over $n\in \N$ is also invariant under  multiplications. In particular, by \eqref{eqn_keythm1}, we get that
\begin{equation*}
	\lim_{N\to\infty}\BEu{n\in [N]} w_\cS(n)a(d^2n)= \alpha_0(\cS)A
\end{equation*}
for any integer $d\ge1$. It follows by taking $N\to\infty$ on both sides of \eqref{eqn_keyprop} that
\begin{equation}
		\lim_{N\to\infty}\frac{1}{N}\sum_{n=1}^N \mu^2(n)w_\cS(n) a(n) = \sum_{d=1}^{D} \frac{\mu(d)w_\cS(d)}{d^2} \cdot \alpha_0(\cS)A+O\Big(\frac{1}{D}\Big).
\end{equation}
Taking $D\to \infty$, we obtain that 
\begin{equation*}
    \lim_{N\to\infty}\frac{1}{N}\sum_{n=1}^N \mu^2(n)w_\cS(n) a(n) = \sum_{d=1}^\infty \frac{\mu(d)w_\cS(d)}{d^2} \cdot \alpha_0(\cS)  A,
\end{equation*}
which is equal to $\frac{\alpha(\cS)}{\zeta(2)}  A$ by \eqref{eqn_ws_zeta2}.\qed

\section{Proof of Theorem~\ref{mainthm_sqfree_applications}}
\label{sec_pf_maincor}

In this section, we prove Theorem~\ref{mainthm_sqfree_applications}. The item (2) in Theorem~\ref{mainthm_sqfree_applications} follows immediately by applying \cite[Theorem~B and Corollary~1.16 ]{BergelsonRichter2022} to Theorem~\ref{mainthm_sqfree}. Hence it suffices to show item (1). 

\noindent\textit{Proof of Theorem~\ref{mainthm_sqfree_applications} (1)}. First, for $\Omega(\varphi_k(n))$, by Theorem~\ref{thm_BKW} we have the following Erd\H{o}s-Kac type theorem established by Bassily, K\'atai, and Wijsmuller:	for any $F\in C_c(\R)$, we have 
	\begin{equation*}
	\lim_{N \to \infty} \frac1N\sum_{n=1}^{N} F\Big( \frac{\Omega(\varphi_k(n)) - a_k(\log \log N)^{k+1} }{b_k(\log \log N)^{k+1/2}} \Big)
	=
	\frac{1}{\sqrt{2\pi}} \int_{-\infty}^{\infty} F(t) e^{-t^2/2} \, dt.
\end{equation*}

Secondly, similar to the proof of \cite[Theorem~1.4]{WWYY2023}, applying Bergelson and Richter's  techniques in \cite{BergelsonRichter2022},  we can obtain the following Loyd type theorem: 
\begin{align}
	&\lim_{N\to\infty}\frac1N\sum_{\substack{1\leq n \leq N\\ p_{\max}(n)\in \cT}} F\Big(\frac{\Omega(\varphi_k(n)) - a_k(\log \log N)^{k+1} }{b_k(\log \log N)^{k+1/2}} \Big)  f(T^{\Omega(n) }x) \nonumber\\
	&	= \delta(\cT)\Big(\frac{1}{\sqrt{2\pi}} \int_{-\infty}^{\infty} F(t) e^{-t^2/2} \, dt\Big)\Big(\int_X f \, d\mu \Big) \label{eqn_maincor_bkw_br}
\end{align}
holds for any $f\in C(X)$ and $x\in X$.

Thirdly, putting 
$a(n)=1_{p_{\max}(n)\in \cT} F\Big(\frac{\Omega(\varphi_k(n)) - a_k(\log \log N)^{k+1} }{b_k(\log \log N)^{k+1/2}} \Big)  f(T^{\Omega(n) }x)$ with $a_k$ and $b_k$ defined as in \eqref{akbk}, we can show that 
\begin{equation*}
	\lim_{N\to\infty}\frac1N\sum_{n=1}^N a(mn)= \lim_{N\to\infty}\frac1N\sum_{n=1}^N a(n)
\end{equation*}
for any $m\in \N$. Indeed, by \cite[Lemma~4.5]{Wang2022}, we have that 
\begin{equation}\label{eqn_maincor_pf_pmax}
	\BEu{n\in[N]} a(mn)= \BEu{n\in[N]} 1_{p_{\max}(n)\in \cT} F\Big(\frac{\Omega(\varphi_k(mn)) - a_k(\log \log N)^{k+1} }{b_k(\log \log N)^{k+1/2}} \Big)  f(T^{\Omega(mn) }x)+O(N^{-c})
\end{equation}
for some constant $c>0$ depending only on $m$. Now, for $\Omega(\varphi_k(mn))$ we have the following lemma.

\begin{lemma}
\label{lem_keylemma} 
Let $k\ge0$ and $m\ge1$. Then we have
\begin{equation*}
    \Omega(\varphi_k(mn))=\Omega(\varphi_k(n))+B
\end{equation*}
with $B=O_{k,m}(1)$ for all $n\ge1$. 
\end{lemma}
\begin{proof}
If $k=0$, then $\varphi_0(n)=n$ and $\Omega(mn)=\Omega(n)+\Omega(m)$, one may take $B=\Omega(m)=O_m(1)$. Suppose $k\ge1$.
For Euler's totient function $\varphi$, we have the following identity
\[
\varphi(mn)=\varphi(n)\cdot\frac{\varphi(m)\gcd(m,n)}{\varphi(\gcd(m,n))}.
\]
It follows that $\varphi(mn)=\varphi(n)B_1$ with $B_1\mid\varphi(m)m$. Let $\tilde\varphi(m)=\varphi(m)m$. One may easily check that the following property holds: $\tilde\varphi(a) \mid  \tilde\varphi(b) $ for all $a \mid b$. Then by induction on $k$, we may write $\varphi_k(mn)=\varphi_k(n)B_k$ with $B_k\mid\tilde\varphi_k(m)$, where $\tilde\varphi_k$ is the $k$-fold iterate of $\tilde\varphi$. It follows that 
$\Omega(\varphi_k(mn))=\Omega(\varphi_k(n))+B$ with $B=\Omega(B_k)\le\Omega(\tilde\varphi_k(m))$. This completes the proof of the lemma.
\end{proof}

Since $\Omega(\varphi_k(mn))=\Omega(\varphi_k(n))+O_{k,m}(1)$, for $n\leq N$,
\[
F\Big(\frac{\Omega(\varphi_k(mn)) - a_k(\log \log N)^{k+1} }{b_k(\log \log N)^{k+1/2}} \Big)- F\Big(\frac{\Omega(\varphi_k(n)) - a_k(\log \log N)^{k+1} }{b_k(\log \log N)^{k+1/2}} \Big)=o_{N\to\infty}(1)
\]
due to $F\in C_c(\R)$ is uniformly continuous. Thus, by \eqref{eqn_maincor_pf_pmax} we have that 
\begin{equation*}
	\BEu{n\in[N]} a(mn)= \BEu{n\in[N]} 1_{p_{\max}(n)\in \cT} F\Big(\frac{\Omega(\varphi_k(n)) - a_k(\log \log N)^{k+1} }{b_k(\log \log N)^{k+1/2}} \Big)  f(T^{\Omega(mn) }x)+o_{N\to\infty}(1).
\end{equation*}
By \eqref{eqn_maincor_bkw_br}, we get that 
\begin{multline*}
        \lim_{N\to\infty} \BEu{n\in[N]} 1_{p_{\max}(n)\in \cT} F\Big(\frac{\Omega(\varphi_k(n)) - a_k(\log \log N)^{k+1} }{b_k(\log \log N)^{k+1/2}} \Big)  f(T^{\Omega(n) }\circ T^{\Omega(m)}x)\\
   = \lim_{N\to\infty} \BEu{n\in[N]} a(n),
\end{multline*}
which implies that
\begin{equation*}
    \lim_{N\to \infty}\BEu{n\in[N]} a(mn)=\lim_{N\to\infty} \BEu{n\in[N]} a(n).
\end{equation*}

Therefore, $a(n)$ is of invariant average under multiplications. Hence \eqref{eqn_maincor_bkw} follows by Theorem~\ref{mainthm_sqfree}.\qed

\section{Proof of Theorem~\ref{mainthm_kfull}}
\label{sec_mainthm_proof}

To prove Theorem~\ref{mainthm_kfull}, we first prove the following proposition. It gives a relation between the average over $k$-full numbers and the standard average.

\begin{proposition}
\label{mainprop}
Let $k\ge2$. Let $a: \N\to\C$ be  bounded. Then for any $1\leq D_1\leq N^{\frac1{(k-1)(k+1)}}, \dots, 1\leq D_{k-1}\leq N^{\frac1{(k-1)(2k-1)}}$, we have
\begin{multline}\label{eqn_mainprop}
	\frac1{N^{\frac1k}}\sum_{\substack{1\leq n \leq N\\ n \,\text{is} \, k\text{-full}}} a(n) = \sum_{\substack{n_1\leq D_1,\dots, n_{k-1}\leq D_{k-1}\\ (n_i, n_j)=1, \,\forall i<j}}  \frac{\mu^2(n_1)\cdots\mu^2(n_{k-1}) }{n_1^{1+\frac1k}\cdots n_{k-1}^{1+\frac{k-1}k}}  \cdot \\ \BEu{m\in [V_k] } a(m^k n_1^{k+1}\cdots n_{k-1}^{2k-1})
		+ O\of{D_1^{-\frac1{k}}}+\cdots + O\of{ D_{k-1}^{-\frac{k-1}k}} + O\big(N^{-\frac1{k(k+1)}}\big), 
\end{multline}
where 
\[
V_k=\frac{N^{\frac1k}}{n_1^{1+\frac1k}\cdots n_{k-1}^{1+\frac{k-1}k}},
\]
and the implied constants depend on $k$ and $\sup_n|a(n)|$ only.
\end{proposition}

\begin{proof}
	For any $k$-full number $n$, we write it as
$$n=m^k n_1^{k+1}\cdots n_{k-1}^{2k-1},$$
where $n_1,\dots, n_{k-1}$ are squarefree and pairwise coprime. This expression is unique. So we can rewrite the summation over $k$-full numbers as $k$ summations over $m, n_1,\dots, n_{k-1}$:
\begin{align}
	&\qquad\frac1{N^{\frac1k}}\sum_{\substack{1\leq n \leq N\\ n \,\text{is} \, k\text{-full}}} a(n) \nonumber\\
	& =  \frac1{N^{\frac1k}}\sum_{\substack{m^k n_1^{k+1}\cdots n_{k-1}^{2k-1} \leq N\\  \mu^2(n_1)=\cdots=\mu^2(n_{k-1})=1\\ (n_i, n_j)=1, \,\forall i<j}} a(m^k n_1^{k+1}\cdots n_{k-1}^{2k-1})  \nonumber\\
	&= \frac1{N^{\frac1k}}\sum_{\substack{n_1^{k+1}\cdots n_{k-1}^{2k-1} \leq N\\  \mu^2(n_1)=\cdots=\mu^2(n_{k-1})=1\\ (n_i, n_j)=1, \,\forall i<j}} \sum_{m\leq V_k } a(m^k n_1^{k+1}\cdots n_{k-1}^{2k-1}) \nonumber\\
	&= \sum_{\substack{n_1^{k+1}\cdots n_{k-1}^{2k-1} \leq N\\ (n_i, n_j)=1, \,\forall i<j}}  \frac{\mu^2(n_1)\cdots\mu^2(n_{k-1}) }{n_1^{1+\frac1k}\cdots n_{k-1}^{1+\frac{k-1}k}}  \cdot \frac1{V_k}\sum_{m\leq V_k } a(m^k n_1^{k+1}\cdots n_{k-1}^{2k-1}) \nonumber\\
	&= \sum_{\substack{n_1^{k+1}\cdots n_{k-1}^{2k-1} \leq N\\ (n_i, n_j)=1, \,\forall i<j}}  \frac{\mu^2(n_1)\cdots\mu^2(n_{k-1}) }{n_1^{1+\frac1k}\cdots n_{k-1}^{1+\frac{k-1}k}}  \cdot \BEu{m\in [V_k] } a(m^k n_1^{k+1}\cdots n_{k-1}^{2k-1})\nonumber\\
	& - \sum_{\substack{n_1^{k+1}\cdots n_{k-1}^{2k-1} \leq N\\ (n_i, n_j)=1, \,\forall i<j}}  \frac{\mu^2(n_1)\cdots\mu^2(n_{k-1}) }{n_1^{1+\frac1k}\cdots n_{k-1}^{1+\frac{k-1}k}}  \cdot \frac{\{V_k\}}{V_k}\BEu{m\in [V_k] } a(m^k n_1^{k+1}\cdots n_{k-1}^{2k-1}) \nonumber\\
	&\colonequals S_1-S_2, \label{eqn_mainthm_pf}
\end{align}
where $\{V_k\}$ denotes the fractional part of $V_k$.

Notice that
$$\BEu{m\in [V_k]} a(m^k n_1^{k+1}\cdots n_{k-1}^{2k-1})=O(1). $$
For $S_1$,  we divide it into two parts according to $n_{k-1}\leq D_{k-1}$ and $n_{k-1}>D_{k-1}$:
\begin{align}
		S_1&= \sum_{\substack{n_1^{k+1}\cdots n_{k-1}^{2k-1} \leq N, n_{k-1}\leq D_{k-1}\\ (n_i, n_j)=1, \,\forall i<j}}  \frac{\mu^2(n_1)\cdots\mu^2(n_{k-1}) }{n_1^{1+\frac1k}\cdots n_{k-1}^{1+\frac{k-1}k}}  \cdot \BEu{m\in [V_k]} a(m^k n_1^{k+1}\cdots n_{k-1}^{2k-1})\nonumber\\
		&\quad+ \sum_{\substack{n_1^{k+1}\cdots n_{k-1}^{2k-1} \leq N, n_{k-1}> D_{k-1}\\ (n_i, n_j)=1, \,\forall i<j}}  \frac{\mu^2(n_1)\cdots\mu^2(n_{k-1}) }{n_1^{1+\frac1k}\cdots n_{k-1}^{1+\frac{k-1}k}}  \cdot \BEu{m\in [V_k]} a(m^k n_1^{k+1}\cdots n_{k-1}^{2k-1}). \label{eqn_S1_1}
\end{align}

Since the series over $n_1, \dots, n_{k-2}$ is absolutely convergent, the second term in \eqref{eqn_S1_1} is bounded by
$$\sum_{n_{k-1}>D_{k-1}} \frac1{n_{k-1}^{1+\frac{k-1}k}}\ll D_{k-1}^{-\frac{k-1}k}.$$
This implies that
\begin{multline*}
	S_1= \sum_{\substack{n_1^{k+1}\cdots n_{k-1}^{2k-1} \leq N, n_{k-1}\leq D_{k-1}\\ (n_i, n_j)=1, \,\forall i<j}}  \frac{\mu^2(n_1)\cdots\mu^2(n_{k-1}) }{n_1^{1+\frac1k}\cdots n_{k-1}^{1+\frac{k-1}k}}  \cdot \BEu{m\in [V_k]} a(m^k n_1^{k+1}\cdots n_{k-1}^{2k-1}) \\ + O\of{ D_{k-1}^{-\frac{k-1}k}}.
\end{multline*}

By induction, we obtain an estimation for $S_1$ as follows
\begin{multline}\label{eqn_mainthm_pf_S1}
		S_1=\sum_{\substack{n_1\leq D_1,\dots, n_{k-1}\leq D_{k-1}\\ (n_i, n_j)=1, \,\forall i<j}}  \frac{\mu^2(n_1)\cdots\mu^2(n_{k-1}) }{n_1^{1+\frac1k}\cdots n_{k-1}^{1+\frac{k-1}k}}  \cdot \BEu{m\in [V_k]} a(m^k n_1^{k+1}\cdots n_{k-1}^{2k-1})\\
		+ O\of{D_1^{-\frac1{k}}}+\cdots + O\of{ D_{k-1}^{-\frac{k-1}k}}.
\end{multline}

As regards  $S_2$, we have
\[
S_2=O\Big(\frac1{N^{\frac{1}{k}}} \sum_{n_1^{k+1}\cdots n_{k-1}^{2k-1} \leq N}1\Big),
\]
and the inner summation in above $O$-term is bounded by $N^{\frac{1}{k+1}}$. In fact,
\begin{align}
	\sum_{n_1^{k+1}\cdots n_{k-1}^{2k-1} \leq N}1 &=\sum_{n_2^{k+2}\cdots n_{k-1}^{2k-1} \leq N}\sum_{n_1\leq\frac{N^{\frac1{k+1}}}{n_2^{1+\frac1{k+1}}\cdots n_{k-1}^{1+\frac{k-2}{k+1}}}}1 \nonumber\\
	&\leq \sum_{n_2^{k+2}\cdots n_{k-1}^{2k-1} \leq N}  \frac{N^{\frac1{k+1}}}{n_2^{1+\frac1{k+1}}\cdots n_{k-1}^{1+\frac{k-2}{k+1}}} \nonumber\\
	&\leq N^{\frac1{k+1}}\sum_{n_2=1}^\infty\cdots\sum_{n_{k-1}=1}^\infty \frac{1}{n_2^{1+\frac1{k+1}}\cdots n_{k-1}^{1+\frac{k-2}{k+1}}} \nonumber\\
	&\ll N^{\frac1{k+1}}. \label{eqn_mainprop_pf_S2-1}
\end{align}
We remark that Vogts \cite{Vogts1985} gave an asymptotic formula for the counting function on numbers of the general form $n_1^{l_1}\cdots n_{r}^{l_r}$, from which \eqref{eqn_mainprop_pf_S2-1} also follows, where $1\leq l_1\leq \cdots\leq l_r$ are real numbers. Plugging \eqref{eqn_mainprop_pf_S2-1} into $S_2$, we get
\begin{equation}\label{eqn_mainthm_pf_S2}
	S_2=O\big(N^{-\frac1{k(k+1)}}\big).
\end{equation}

Thus, \eqref{eqn_mainprop} follows immediately by combining  \eqref{eqn_mainthm_pf}, \eqref{eqn_mainthm_pf_S1} and  \eqref{eqn_mainthm_pf_S2} together.
\end{proof}

\begin{proof}[Proof of Theorem~\ref{mainthm_kfull}] To prove Theorem~\ref{mainthm_kfull}, it is equivalent to show that
\begin{equation}\label{eqn_mainthm_equiv}
	\lim_{N\to\infty}\frac1{N^{\frac1k}}\sum_{\substack{1\leq n \leq N\\ n \,\text{is} \, k\text{-full}}} a(n)=  c_k \lim_{N\to\infty}\BEu{n\in  [N] } a(n^k),
\end{equation}
where $c_k=\prod_{p}(1+\sum_{m=k+1}^{2k-1}p^{-m/k})$ as in \eqref{eqn_ErdosSzekeres1934}.

By Proposition~\ref{mainprop}, \eqref{eqn_mainprop} holds for any $1\leq D_1\leq N^{\frac1{(k-1)(k+1)}}, \dots, 1\leq D_{k-1}\leq N^{\frac1{(k-1)(2k-1)}}$. Fixing $D_1,\dots, D_{k-1}$, we take $N\to\infty$ in \eqref{eqn_mainprop} first. Since the sequence $a(n)$ is of $k$-th power invariant average under multiplication, we have
\begin{equation*}
	\lim_{N\to\infty}\BEu{m\in [V_k] } a(m^k n_1^{k+1}\cdots n_{k-1}^{2k-1}) =  \lim_{N\to\infty}\BEu{m\in [N] } a(m^k).
\end{equation*}
This implies that
\begin{multline}\label{eqn_mainthm_pf_keyidentity2}
	\lim_{N\to\infty} \frac1{N^{\frac1k}}\sum_{\substack{1\leq n \leq N\\ n \,\text{is} \, k\text{-full}}} a(n)= \sum_{\substack{n_1\leq D_1,\dots, n_{k-1}\leq D_{k-1}\\ (n_i, n_j)=1, \,\forall i<j}}  \frac{\mu^2(n_1)\cdots\mu^2(n_{k-1}) }{n_1^{1+\frac1k}\cdots n_{k-1}^{1+\frac{k-1}k}} \cdot   \lim_{N\to\infty}\BEu{m\in [N] } a(m^k)\\
	 + O\of{D_1^{-\frac1{k}}}+\cdots + O\of{ D_{k-1}^{-\frac{k-1}k}}.
\end{multline}

Now, taking $D_1, \dots, D_{k-1} \to\infty$ in \eqref{eqn_mainthm_pf_keyidentity2}, we get that
\begin{equation}\label{eqn_mainthm_pf_keyidentity3}
	\lim_{N\to\infty} \frac1{N^{\frac1k}}\sum_{\substack{1\leq n \leq N\\ n \,\text{is} \, k\text{-full}}} a(n)=\mathop{\sum_{n_1=1}^\infty\cdots\sum_{n_{k-1}=1}^\infty}_{(n_i, n_j)=1, \,\forall   i<j}  \frac{\mu^2(n_1)\cdots\mu^2(n_{k-1}) }{n_1^{1+\frac1k}\cdots n_{k-1}^{1+\frac{k-1}k}} \cdot   \lim_{N\to\infty}\BEu{m\in [N] } a(m^k).
\end{equation}

By \eqref{eqn_ErdosSzekeres1934}, taking $a(n)=1$ for all $n$ in \eqref{eqn_mainthm_pf_keyidentity3} gives
\begin{equation}\label{eqn_mainthm_pf_ck}
	\mathop{\sum_{n_1=1}^\infty\cdots\sum_{n_{k-1}=1}^\infty}_{(n_i, n_j)=1, \,\forall   i<j}  \frac{\mu^2(n_1)\cdots\mu^2(n_{k-1}) }{n_1^{1+\frac1k}\cdots n_{k-1}^{1+\frac{k-1}k}}= \lim_{N\to\infty} \frac1{N^{\frac1k}}\sum_{\substack{1\leq n \leq N\\ n \,\text{is} \, k\text{-full}}}1=c_k.
\end{equation}
Therefore, \eqref{eqn_mainthm_equiv} follows by \eqref{eqn_mainthm_pf_keyidentity3}  and \eqref{eqn_mainthm_pf_ck}. This completes the proof of Theorem~\ref{mainthm_kfull}.
\end{proof}

\section{Proof of Theorem~\ref{mainthm_kfull_applications}}
\label{sec_proofs}

In this section, we apply Theorem~\ref{mainthm_kfull} to prove Theorem~\ref{mainthm_kfull_applications}. Let $k\ge2$. By Theorem~\ref{mainthm_kfull}, if $a: \N\to\C$ is a bounded function satisfying \eqref{eqn_k_invariant_cond_1} and \eqref{eqn_k_invariant_cond_2}, then 
\begin{equation}\label{eqn_proofs_keyidentity}
\lim_{N\to\infty} \BEu{\substack{ n \in [N]\\ n \,\text{is} \, k\text{-full}}} a(n)= \lim_{N\to\infty}\BEu{n\in  [N] } a(n^k).
\end{equation}
We will show that the arithmetic functions in the summations of \eqref{eqn_BR2022thmA}, \eqref{eqn_EK_origin} and \eqref{eqn_Loyd2022_origin} satisfy properties \eqref{eqn_k_invariant_cond_1} and \eqref{eqn_k_invariant_cond_2}. So we can apply \eqref{eqn_proofs_keyidentity}  for them.

\subsection{Proof of \texorpdfstring{\eqref{eqn_mainthm_ergogic}}{}}  

For the totally uniquely ergodic system $(X,\mu, T)$, any $f\in C(X)$ and any $x\in X$,  we take $a(n)=f(T^{\Omega(n)}x)$.  Then $a(n^k)=f\big((T^k)^{\Omega(n)}x\big)$. Since $(X,\mu, T)$ is totally uniquely ergodic, by definition $(X,\mu, T^k)$ is uniquely ergodic. Thus, by Bergelson-Richter's theorem \eqref{eqn_BR2022thmA}, we have
\begin{equation}\label{eqn_mainthm_ergogic_pf1}
	\lim_{N\to\infty}\BEu{n\in  [N] } a(n^k)=\int_X f \, d\mu.
\end{equation}
Moreover, for any $m\in \N$, we have $a(n^km)=f\big((T^k)^{\Omega(n)} \cdot T^{\Omega(m)}x\big)$. We may take $ T^{\Omega(m)}x$ as an initial point in $X$. It follows by \eqref{eqn_mainthm_ergogic_pf1} that
\begin{equation*}
	\lim_{N\to\infty}\BEu{n\in  [N] } a(n^km)=\int_X f \, d\mu.
\end{equation*}
Hence $a(n)$ is of $k$-th power invariant average under multiplications. By \eqref{eqn_proofs_keyidentity} and \eqref{eqn_mainthm_ergogic_pf1} we get that 
\begin{equation*}
	\lim_{N\to\infty} \BEu{\substack{ n \in [N]\\ n \,\text{is} \, k\text{-full}}} a(n)=\int_X f \, d\mu,
\end{equation*}
which is \eqref{eqn_mainthm_ergogic}. 

\subsection{Proof of \texorpdfstring{\eqref{eqn_mainthm_EK}}{}}  

For any $F\in C_c(\R)$ and $N\ge1$, take 
$$a(n)=F\Big( \frac{\Omega(n) - k\log \log N }{k\sqrt{\log \log N}}\Big).$$
Then
$$a(n^k)=F\Big( \frac{\Omega(n) - \log \log N }{\sqrt{\log \log N}}\Big).$$
By Erd\H{o}s-Kac theorem \eqref{eqn_EK_origin}, we have
\begin{equation}\label{eqn_mainthm_EK_pf1}
	\lim_{N\to\infty}\BEu{n\in  [N] } a(n^k)=\lim_{N\to\infty}\BEu{n\in  [N] } F\Big( \frac{\Omega(n) - \log \log N }{\sqrt{\log \log N}}\Big)=\frac{1}{\sqrt{2\pi}} \int_{-\infty}^{\infty} F(t) e^{-t^2/2} \, dt.
\end{equation}

Moreover, for any $m\in \N$, we have
$$a(n^km)=F\Big( \frac{\Omega(n) +\Omega(m)/k - \log \log N }{\sqrt{\log \log N}}\Big).$$
Since $F$ is a continuous function of compact support, it is uniformly continuous. It follows that
\begin{equation}\label{eqn_mainthm_EK_pf_diff}
	F\Big( \frac{\Omega(n) +\Omega(m)/k - \log \log N }{\sqrt{\log \log N}}\Big)= F\Big( \frac{\Omega(n) - \log \log N }{\sqrt{\log \log N}}\Big) + o_{N\to\infty}(1),
\end{equation}
i.e., $a(n^km)=a(n^k)+ o_{N\to\infty}(1)$. This implies that
\begin{equation*}
	\lim_{N\to\infty}\BEu{n\in  [N] } a(n^km)= \lim_{N\to\infty}\BEu{n\in  [N] } a(n^k).
\end{equation*}
Thus, $a(n)$ is of $k$-th power invariant average under multiplications. By \eqref{eqn_proofs_keyidentity} and \eqref{eqn_mainthm_EK_pf1} we get that
	\begin{equation*}
	\lim_{N \to \infty}  \BEu{\substack{ n \in [N]\\ n \,\text{is} \, k\text{-full}}} a(n)=\frac{1}{\sqrt{2\pi}} \int_{-\infty}^{\infty} F(t) e^{-t^2/2} \, dt,
\end{equation*}
which is \eqref{eqn_mainthm_EK}. 

\subsection{Proof of \texorpdfstring{\eqref{eqn_mainthm_Loyd}}{}}  

Let $N\ge1$. For the totally uniquely ergodic system $(X,\mu, T)$, any $f\in C(X)$, any $x\in X$ and  any $F\in C_c(\R)$, we take 
$$a(n)=F \Big( \frac{\Omega(n) - k\log \log N }{k\sqrt{\log \log N}} \Big) f(T^{\Omega(n) }x).$$
Then
$$a(n^k)=F \Big( \frac{\Omega(n) - \log \log N }{\sqrt{\log \log N}} \Big) f\big((T^k)^{\Omega(n)}x\big).$$
Since $(X,\mu, T^k)$ is uniquely ergodic, by Loyd's theorem \eqref{eqn_Loyd2022_origin} we have
\begin{equation}\label{eqn_mainthm_Loyd_pf1}
	\lim_{N\to\infty}\BEu{n\in  [N] } a(n^k)=\Big(\frac{1}{\sqrt{2\pi}} \int_{-\infty}^{\infty} F(t) e^{-t^2/2} \, dt\Big)\Big( \int_X f \, d\mu \Big).
\end{equation}

Moreover, for any $m\in\N$, 
$$a(n^km)=F\Big( \frac{\Omega(n) +\Omega(m)/k - \log \log N }{\sqrt{\log \log N}}\Big) f\big((T^k)^{\Omega(n)} \cdot T^{\Omega(m)}x\big).$$
By \eqref{eqn_mainthm_EK_pf_diff}, we get that
\begin{equation*}
	a(n^km)= F\Big( \frac{\Omega(n) - \log \log N }{\sqrt{\log \log N}}\Big) f\big((T^k)^{\Omega(n)} \cdot T^{\Omega(m)}x\big) + o_{N\to\infty}(1).
\end{equation*}
It follows by Loyd's theorem again that
\begin{align}
	\lim_{N\to\infty}\BEu{n\in  [N] } a(n^km) & = \lim_{N\to\infty}\BEu{n\in  [N] } F\Big( \frac{\Omega(n) - \log \log N }{\sqrt{\log \log N}}\Big) f\big((T^k)^{\Omega(n)} \cdot T^{\Omega(m)}x\big)\nonumber\\
	&=\Big(\frac{1}{\sqrt{2\pi}} \int_{-\infty}^{\infty} F(t) e^{-t^2/2} \, dt\Big)\Big( \int_X f \, d\mu \Big),\nonumber
\end{align}
which is equal to $\lim_{N\to\infty}\BEu{n\in  [N] } a(n^k)$. Hence $a(n)$ is of $k$-th power invariant average under multiplications. By \eqref{eqn_proofs_keyidentity} and \eqref{eqn_mainthm_Loyd_pf1} we get that 
\begin{equation*}
	\lim_{N\to\infty} \BEu{\substack{ n \in [N]\\ n \,\text{is} \, k\text{-full}}} a(n)=\Big(\frac{1}{\sqrt{2\pi}} \int_{-\infty}^{\infty} F(t) e^{-t^2/2} \, dt\Big)\Big( \int_X f \, d\mu \Big),
\end{equation*}
which is \eqref{eqn_mainthm_Loyd}. This completes the proof of Theorem~\ref{mainthm_kfull_applications}.

\section{Proof of Theorem~\ref{thm_Richter_sqfree}}\label{sec_Richter}

In this section, similar to the proof of Theorems~\ref{mainthm_sqfree} and \ref{mainthm_kfull}, we will apply Propositions~\ref{prop_sqfree} and \ref{mainprop} to prove \eqref{eqn_Richter_sqfree} and \eqref{eqn_Richter_kfull}, respectively. We will cite Richter's theorem in the following proofs: for any bounded $a:\N\to\C$ one has
\begin{equation}\label{eqn_Richter_copy}
	\frac1N\sum_{n\leq N}a(\Omega(n))= \frac1N\sum_{n\leq N}a(\Omega(n)+1)+o_{N\to\infty}(1).
\end{equation}

\subsection{Proof of \texorpdfstring{\eqref{eqn_Richter_sqfree}}{}}

By Proposition~\ref{prop_sqfree}, for any $1\leq D \leq \sqrt{N}$ we have 
\begin{align}
		\frac{1}{N}\sum_{n=1}^N \mu^2(n) a(\Omega(n)) &= \sum_{d=1}^D \frac{\mu(d)}{d^2}\BEu{n\in [\frac{N}{d^2}]}a(\Omega(d^2n))+O\Big(\frac{1}{D}\Big)+O\Big(\frac{1}{\sqrt{N}}\Big), \label{eqn_Richter_sqfree1}\\
		\frac{1}{N}\sum_{n=1}^N \mu^2(n) a(\Omega(n)+1) &= \sum_{d=1}^D \frac{\mu(d)}{d^2}\BEu{n\in [\frac{N}{d^2}]}a(\Omega(d^2n)+1)+O\Big(\frac{1}{D}\Big)+O\Big(\frac{1}{\sqrt{N}}\Big). \label{eqn_Richter_sqfree2}
\end{align}
The difference between \eqref{eqn_Richter_sqfree1} and \eqref{eqn_Richter_sqfree2} is equal to
\begin{align}
	&\quad\frac{1}{N}\sum_{n=1}^N \mu^2(n) a(\Omega(n)) - \frac{1}{N}\sum_{n=1}^N \mu^2(n) a(\Omega(n)+1) \nonumber\\
	&=\sum_{d=1}^D\frac{\mu(d)}{d^2}  \Big(\BEu{n\in [\frac{N}{d^2}]}a(\Omega(n)+\Omega(d^2))- \BEu{n\in [\frac{N}{d^2}]}a(\Omega(n)+\Omega(d^2)+1) \Big)+O\Big(\frac{1}{D}\Big)+O\Big(\frac{1}{\sqrt{N}}\Big).\nonumber
\end{align}
Fix $D$ and take $N\to\infty$ first. By \eqref{eqn_Richter_copy}, we have
\begin{equation*}
	\BEu{n\in [\frac{N}{d^2}]}a(\Omega(n)+\Omega(d^2))- \BEu{n\in [\frac{N}{d^2}]}a(\Omega(n)+\Omega(d^2)+1)= o_{N\to\infty}(1).
\end{equation*}
It follows that
\begin{equation}\label{eqn_Richter_sqfree_pf3}
	\frac{1}{N}\sum_{n=1}^N \mu^2(n) a(\Omega(n)) - \frac{1}{N}\sum_{n=1}^N \mu^2(n) a(\Omega(n)+1)=O\Big(\frac{1}{D}\Big)+ o_{N\to\infty}(1).
\end{equation}
Then \eqref{eqn_Richter_sqfree} follows by taking $D\to\infty$ in \eqref{eqn_Richter_sqfree_pf3}.

\subsection{Proof of \texorpdfstring{\eqref{eqn_Richter_kfull}}{}}

Let $k\ge2$. By Proposition~\ref{mainprop}, for any $1\leq D_1\leq N^{\frac1{(k-1)(k+1)}}, \dots, 1\leq D_{k-1}\leq N^{\frac1{(k-1)(2k-1)}}$, we have
\begin{multline}\label{eqn_Richter_kfull1}
	\frac1{N^{\frac1k}}\sum_{\substack{1\leq n \leq N\\ n \,\text{is} \, k\text{-full}}} a(\Omega(n)) = \sum_{\substack{n_1\leq D_1,\dots, n_{k-1}\leq D_{k-1}\\ (n_i, n_j)=1, \,\forall i<j}}  \frac{\mu^2(n_1)\cdots\mu^2(n_{k-1}) }{n_1^{1+\frac1k}\cdots n_{k-1}^{1+\frac{k-1}k}} \\ \cdot \BEu{m\in [V_k] } a(\Omega(m^k n_1^{k+1}\cdots n_{k-1}^{2k-1})) + O\of{D_1^{-\frac1{k}}}+\cdots + O\of{ D_{k-1}^{-\frac{k-1}k}} + O\big(N^{-\frac1{k(k+1)}}\big), 
\end{multline}
and 
\begin{multline}\label{eqn_Richter_kfull2}
	\frac1{N^{\frac1k}}\sum_{\substack{1\leq n \leq N\\ n \,\text{is} \, k\text{-full}}} a(\Omega(n)+k) = \sum_{\substack{n_1\leq D_1,\dots, n_{k-1}\leq D_{k-1}\\ (n_i, n_j)=1, \,\forall i<j}}  \frac{\mu^2(n_1)\cdots\mu^2(n_{k-1}) }{n_1^{1+\frac1k}\cdots n_{k-1}^{1+\frac{k-1}k}} \\ \cdot \BEu{m\in [V_k] } a(\Omega(m^k n_1^{k+1}\cdots n_{k-1}^{2k-1})+k) + O\of{D_1^{-\frac1{k}}}+\cdots + O\of{ D_{k-1}^{-\frac{k-1}k}} + O\big(N^{-\frac1{k(k+1)}}\big).
\end{multline}

The difference between \eqref{eqn_Richter_kfull1} and \eqref{eqn_Richter_kfull2} is equal to
\begin{multline*}
	\frac1{N^{\frac1k}}\sum_{\substack{1\leq n \leq N\\ n \,\text{is} \, k\text{-full}}} a(\Omega(n)) - \frac1{N^{\frac1k}}\sum_{\substack{1\leq n \leq N\\ n \,\text{is} \, k\text{-full}}} a(\Omega(n)+k) 
	= \sum_{\substack{n_1\leq D_1,\dots, n_{k-1}\leq D_{k-1}\\ (n_i, n_j)=1, \,\forall i<j}}  \frac{\mu^2(n_1)\cdots\mu^2(n_{k-1}) }{n_1^{1+\frac1k}\cdots n_{k-1}^{1+\frac{k-1}k}} \\ 
	\cdot \left(\BEu{m\in [V_k] } a(\Omega(m^k n_1^{k+1}\cdots n_{k-1}^{2k-1})) - \BEu{m\in [V_k]} a(\Omega(m^k n_1^{k+1}\cdots n_{k-1}^{2k-1})+k) \right)\\
	+ O\of{D_1^{-\frac1{k}}}+\cdots + O\of{ D_{k-1}^{-\frac{k-1}k}} + O\big(N^{-\frac1{k(k+1)}}\big).
\end{multline*}

Fix $D_1,\dots, D_{k-1}$ and take $N\to\infty$ first. Let $b(m)=a\big(km+\Omega(n_1^{k+1}\cdots n_{k-1}^{2k-1})\big)$. Then by \eqref{eqn_Richter_copy}, we have
\begin{equation*}
	\BEu{m\in [V_k]} b(\Omega(m)) - \BEu{m\in [V_k] } b(\Omega(m)+1) =o_{N\to\infty}(1).
	\end{equation*}
This implies that
\begin{equation}\label{eqn_Richter_kfull_diff2}
	\frac1{N^{\frac1k}}\sum_{\substack{1\leq n \leq N\\ n \,\text{is} \, k\text{-full}}} a(\Omega(n)) - \frac1{N^{\frac1k}}\sum_{\substack{1\leq n \leq N\\ n \,\text{is} \, k\text{-full}}} a(\Omega(n)+k) 
	= O\of{D_1^{-\frac1{k}}}+\cdots + O\of{ D_{k-1}^{-\frac{k-1}k}} + o_{N\to\infty}(1).
\end{equation}
Then \eqref{eqn_Richter_kfull} follows immediately by taking $D_1 ,\dots, D_{k-1} \to\infty$ in \eqref{eqn_Richter_kfull_diff2}. This completes the proof of Theorem~\ref{thm_Richter_sqfree}.

\section*{Acknowledgments}
The authors are deeply grateful to the referee and the copy editor for a very careful review and a number of precise comments and helpful suggestions, which improve this paper a lot.  The authors would also like to thank Rongzhong Xiao for helpful discussion and suggesting reference \cite{CellarosiSinai2013} to us. This work is supported by the National Natural Science Foundation of China (Grant No. 12561001). Huixi Li’s research is supported by the National Natural Science Foundation of China (Grant No. 12201313). Shaoyun Yi is supported by the National Natural Science Foundation of China (Nos.~12301016, 12471187) and the Fundamental Research Funds for the Central Universities (No.~20720230025).


\end{document}